\documentclass[a4paper,11pt]{amsart}

\usepackage{amssymb}
\usepackage{mathrsfs}
\usepackage{amsmath,amssymb,amsthm,latexsym,amscd,mathrsfs}
\usepackage{indentfirst}
\usepackage{stmaryrd}
\usepackage{graphicx}
\usepackage{extarrows}
\usepackage{amssymb}
\usepackage{amsmath,amssymb,amsthm,latexsym,amscd,mathrsfs}
\usepackage{indentfirst}
\usepackage{multirow}
\usepackage{latexsym}
\usepackage{amsfonts}
\usepackage{color}
\usepackage{pictexwd,dcpic}
\usepackage{graphicx}
\usepackage{psfrag}
\usepackage{hyperref}
\usepackage{comment}

\numberwithin{equation}{section} \theoremstyle{plain}
\newtheorem{thm}{Theorem}[section]

\newtheorem{lem}[thm]{Lemma}

\newtheorem{ack}{Acknowledgements}   
 \makeatletter

\def\<{\langle}
\def\>{\rangle}
\def\({\left(}
\def\){\right)}
\def\[{\left[}
\def\]{\right]}

\makeatother
\title[DDVV inequality for Hermitian matrices]{DDVV-type inequality for Hermitian matrices}
\author[J.Q. Ge]{Jianquan Ge}
\address{School of Mathematical Sciences, Laboratory of Mathematics and Complex Systems, Beijing Normal
University, Beijing 100875, P.R. CHINA.}
\email{jqge@bnu.edu.cn}

\author[S. Xu]{Song Xu}
\email{mike9710@163.com}

\author[H. Y. You]{Hangyu You}
\email{779640002@qq.com}

\author[Y. Zhou]{Yi Zhou}
\email{zyzy8369@163.com}

\date{}

\thanks{The first author is partially supported by the NSFC (No. 11331002, 11522103) and by the Fundamental Research Funds for the Central Universities. }
\begin{document}
\maketitle

\begin{abstract}
In this paper we extend DDVV-type inequalities involving the Frobenius norm of commutators from real symmetric and skew-symmetric matrices to Hermitian and skew-Hermitian matrices.
\end{abstract}

\section{Introduction}\label{introduction}
A DDVV-type inequality is an estimate of the form (cf. \cite{Lu07, LW11, GT11})
\begin{equation}\label{DDVVtypeineq}
\sum^m_{r,s=1}\|[B_r,B_s]\|^2\leq c\(\sum^m_{r=1}\|B_r\|^2\)^2,
\end{equation}
considered for certain $n\times n$ matrices $B_1\ldots,B_m$, where $[A,B]=AB-BA$ is the commutator and
$\|B\|^2=\operatorname{tr}(BB^*)$ is the squared Frobenius norm.
Naturally, one is interested in the best constant $c$ so that (\ref{DDVVtypeineq}) stays valid for all
matrices in the regarded class. The name has its origin in submanifold geometry. The DDVV conjecture concerns the following inequality between the scalar curvature $\rho$ (intrinsic invariant), the mean curvature $H$ and the normal scalar curvature $\rho^\bot$ (extrinsic invariants) of a submanifold $M^n$ in a real space form $N^{n+m}(\kappa)$ with constant sectional curvature $\kappa$ (cf. \cite{DDVV99}):
$$\rho+\rho^\bot\leq\|H\|^2+\kappa.$$
Now, the equivalent algebraic version is
precisely (\ref{DDVVtypeineq}) with the value $c=1$. As the problem has a geometric background,
only real symmetric matrices are taken into account. This claim was shown by Lu \cite{Lu11} and Ge-Tang \cite{GT08} independently and differently.
The matrix tuples giving equality were completely determined and under some rotation and orthogonal congruence, the matrices of such a tuple are all zero except for two in the form of $diag(A_1,0)$ and $diag(A_2,0)$, where $0$ is the zero matrix and
$$A_1:=
\begin{pmatrix}
    \lambda & 0 \\
    0 & -\lambda \\
\end{pmatrix},
\quad
A_2:=
\begin{pmatrix}
    0 & \lambda \\
    \lambda & 0 \\
\end{pmatrix}
$$ with $\lambda\geq0$. The equality also has a geometric interpretation: submanifolds achieving the equality everywhere are called Wintgen ideal submanifolds which are invariant under conformal transformations and are not yet classified so far, although many partial results and studies are available in the literature (cf. \cite{CL08, DT09, XLMW14}, etc.).
Recently, the geometric DDVV inequality was strengthened on the focal submanifolds of isoparametric hypersurfaces in unit spheres. The points attaining the equality were explicitly calculated, in particular, those focal submanifolds which are Wintgen ideal submanifolds were classified in \cite{GTY16}.

 After the real symmetric matrices (which occur in the second fundamental tensor in submanifold geometry), real skew-symmetric matrices were investigated in \cite{Ge14} for the DDVV-type inequality (\ref{DDVVtypeineq}), since they naturally occur in the integrability tensor in Riemannian submersion geometry.
There is an interesting phenomenon of ``dualities": Lu \cite{Lu11} applied the DDVV inequality for symmetric matrices to give Simons type inequality for minimal submanifolds of spheres in submanifold geometry; Ge \cite{Ge14} established the DDVV-type inequality for skew-symmetric matrices and applied it to give a Simons-type inequality for Yang-Mills fields in Riemannian submersion geometry. Simply speaking, the dualities appear in algebra as symmetric to skew-symmetric, and in their applications in geometry as immersion (submanifold) to submersion. The dual phenomenon between the objects (minimal submanifolds and Yang-Mills fields) of applications of algebra to geometry was initially investigated by Tian \cite{Ti00}.
Now, for the skew-symmetric class, when $m\geq3$, $c=1/3$ for $n=3$ and $c=2/3$ for $n\geq4$. Similarly, for obtaining equality all but three matrices must be zero, and these are again given via transforming the simplest representatives of the skew-symmetric class. In the case $m=2$ the best constant can be shown to be even smaller, $c=1/4$ for $n=3$ and $c=1/2$ for $n\geq4$; see \cite{BI05}\footnote{Notice that in \cite{BI05} there is a scaling of the norm by a factor of $\sqrt{2}$ to the Frobenius norm.} (see also Lemma 2.5 of \cite{Ge14}), which showed that
$\|[B_1, B_2]\|^2\leq \tilde{c}\|B_1\|^2\|B_2\|^2 $ with $\tilde{c}=1/2$ for $n=3$ and $\tilde{c}=1$ for $n\geq4$. In fact, this can be seen easily from
$$2\|[B_1, B_2]\|^2\leq 2\tilde{c}\|B_1\|^2\|B_2\|^2\leq \frac{\tilde{c}}{2}(\|B_1\|^2+\|B_2\|^2)^2.$$

Indeed, when $m=2$, the problem is closely related to the B\"{o}ttcher-Wenzel inequality (cf. \cite{BW05,BW08}, see also \cite{VJ08, Lu11, Lu12}): $\|[X, Y]\|^2\leq 2\|X\|^2\|Y\|^2$, which is even true for arbitrary complex matrices (cf. \cite{BW08}).
A unified generalization of the DDVV inequality and the B\"{o}ttcher-Wenzel inequality has been conjectured and ongoingly studied by Lu and Wenzel \cite{LW11}.

In this paper, we extend the DDVV-type inequalities from real matrices to complex matrices.
As in the real case, we consider the complex matrices with symmetries, namely, the Hermitian matrices and the skew-Hermitian matrices.
In fact, since $\mathbf{i} B$ is skew-Hermitian for any Hermitian matrix $B$, the inequality for the skew-Hermitian case is the same as the Hermitian case, which is slightly different from the real case.

Throughout this paper, we put $K:=U(n)\times O(m)$. A $K$ action on a family of matrices $(B_1,\cdots,B_m)$ is given by
$$(P,R)\cdot(B_1,\cdots,B_m):=\left(\sum_{j=1}^mR_{j1}P^*B_jP,\cdots,\sum_{j=1}^mR_{jm}P^*B_jP\right),$$ for $(P,R)\in K$, where $R=(R_{jk})\in O(m)$ acts as a rotation on the matrix tuple $(P^*B_1P,\cdots,P^*B_mP)$.
\begin{thm}\label{DDVVHermthm}
Let $B_1,\cdots,B_m$ be $n\times n$ Hermitian matrices.
\begin{enumerate}
\item If\ $m\geq3$, then $c=\frac43$ in $(\ref{DDVVtypeineq})$, i.e.,
$$
\sum^m_{r,s=1}\|\[B_r,B_s\]\|^2\leq \frac43\(\sum^m_{r=1}\|B_r\|^2\)^2.
$$
The equality holds if and only if under some $K$ action all $B_r$'s are zero except for $3$
matrices in the form of $diag(H_1,0)$, $diag(H_2,0)$ and $diag(H_3,0)$, where
$$H_1:=
\begin{pmatrix}
\lambda & 0 \\
0 & -\lambda \\
\end{pmatrix},
H_2:=
\begin{pmatrix}
0 & \lambda \\
\lambda & 0 \\
\end{pmatrix},
H_3:=
\begin{pmatrix}
0 & -\lambda\mathbf{i} \\
\lambda\mathbf{i} & 0 \\
\end{pmatrix},
$$
with $\lambda\geq0$.
\item If $m=2$, then $c=1$ in $(\ref{DDVVtypeineq})$, i.e.,
$$
\sum^2_{r,s=1}\|\[B_r,B_s\]\|^2\leq \(\sum^2_{r=1}\|B_r\|^2\)^2.
$$
The equality holds if and only if under some $K$ action $B_1=diag(H_1,0)$ and $B_2=diag(\cos \theta H_2+\sin \theta H_3,0)$.
\end{enumerate}

When considering skew-Hermitian matrices $B_r$, the same statements are true with $H_j$ replaced by
$\mathbf{i}H_j$, $j=1,2,3$.
\end{thm}
Clearly, $H_3$ is a natural candidate for extending the real-symmetric matrices spanned by $A_1=H_1$ and $A_2=H_2$ to the Hermitian class. The second inequality of Theorem \ref{DDVVHermthm} is implied by the B\"{o}ttcher-Wenzel inequality
and its proof can be even traced back to Chern-do Carmo-Kobayashi \cite{CDK} for real symmetric matrices; hence we omit its proof.
The equality condition follows simply by restricting the matrices to be Hermitian (cf. \cite{BW08}).
It seems that the first inequality of Theorem \ref{DDVVHermthm} would also hold for arbitrary complex matrices, at least for real matrices.
The reason is that any matrix can be written as the sum of Hermitian (symmetric) and skew-Hermitian (skew-symmetric) matrices,
which enables one to combine their orthonormal bases and use this combined basis to compute the commutators as in (\ref{comm-bases}, \ref{comm-bases1}, \ref{comm-bases2}).
As for possible geometric applications of this DDVV-type inequality for Hermitian matrices, we believe that it would also derive a Simons-type inequality in K\"{a}hler (complex) geometry with certain ``Hermitian" tensors instead of the usual second fundamental tensor used in submanifold geometry known from the real case (cf. \cite{Lu11}).

The proof follows the method of Ge and Tang \cite{GT08}, \cite{Ge14} and we refer to \cite{GT11} for a sketch. Here we introduce the key ideas simply to follow this method.
First of all, one needs to ``find" (or guess) the best constant $c$ by testing examples of matrices in the regarded class. The number of the extreme matrices should not be too big (e.g. $2$, $3$ or $4$, etc.) and certain symmetries would exist as in the known DDVV-type inequalities. After the DDVV-type inequality (\ref{DDVVtypeineq}) has been transformed into  the non-positivity of a quadratic form $f_Q(x)$ ($x\in\mathbb{R}^N_+$, $Q\in SO(N)$) as (\ref{quadratic-form}) which is expected to be negative in the interior of $\mathbb{R}^N_+$ for any $Q\in SO(N)$, one has to verify three conditions (see (a) to (c) on page 11) in the approximation procedure of the proof. The condition (b) is automatically satisfied, while (a) and (c) (in particular, the inequality (\ref{contrad})) require one to prepare the formulae (\ref{eqno2}, \ref{eqno3}) and the preparatory lemmas (e.g. Lemmas \ref{lem3}, \ref{lem4}). The calculations of these preliminary results are similar to the original proof for symmetric matrices but more involved, as the regarded matrix class varies. One would succeed if the steps can be done, otherwise, one needs to try another $c$ by testing more examples. In fact, one could deduce some clues for finding the best constant $c$ from the estimates in the Lemmas \ref{lem3} and \ref{lem2} where $\frac 43$ would have to be replaced by $c$.

The paper is organized as follows: in Section \ref{sec2} we give all preliminary results mentioned above; in Section \ref{sec3} we firstly transfer the inequality to the non-positivity of the quadratic form $f_Q(x)$ and then show (a-c) with the help of the preliminary results from Section \ref{sec2}.
For those who have seen the method before, Section \ref{sec2} is the interesting part. If you are new to the method, it could be better to skip the technical lemmas and to read Section \ref{sec3} first with Section \ref{sec2} as a reference.
\section{Notations and preparatory lemmas}\label{sec2}
This section contains the necessary adaptions of the core lemmas from the DDVV proof given in \cite{GT08} to the new constant we aim at. We denote the space of $m\times n$ real matrices by $M(m,n)$, the space of $n\times n$ real matrices by $M(n)$, and the space of $n\times n$ Hermitian matrices by $HM(n)$, which has dimension $N:=n^2$. All spaces are regarded as real vector spaces.

For every $(i,j)$ with $1\leq i,j\leq n$, let
$$
\check{E}_{ij}:=\
\begin{cases}
E_{ii}                                 &\text{if}\ i=j,\\
(E_{ij}+E_{ji})/\sqrt{2}               &\text{if}\ i<j,\\
\mathbf{i}(E_{ij}-E_{ji})/\sqrt{2}      &\text{if}\ i>j,
\end{cases}
$$
where $E_{ij}\in M(n)$ is the matrix with 1 in position $(i,j)$ and 0 elsewhere. Clearly $\{\check{E}_{ij}\}$ is an orthonormal basis of $HM(n)$. The third component gives new basis elements that were absent in the real-symmetric proof of \cite{GT08}. The skew-symmetric case in \cite{Ge14} regarded the natural basis spanned by $(E_{ij}-E_{ji})/\sqrt{2}$ at this place. Let us put an order on the index set $S:=\{(i,j)\ |\ 1\leq i,j\leq n\}$ by
\begin{equation}\label{eqno1}
(i,j)<(k,l)\Longleftrightarrow i<k, \text{or} \ i=k\ \text{and} \ j<l.
\end{equation}
We use this order to index elements of $S$ with a single (Greek) index in the range $\{1,...,N\}$. Clearly, $(i,j)\mapsto \alpha=(i-1)n+j$ can be inverted.

For $\alpha \hat{=}(i,j)\leq \beta \hat{=}(k,l)$ in $S$, direct calculations imply
\begin{equation}\label{eqno2}
\left\|\[\check{E}_\alpha,\check{E}_\beta\]\right\|^2=
\begin{cases}
2         &\text{if}\ i=l<j=k,\\
\\
1         &\text{if}\ i=j=k<l\ |\ i<j=k=l\ |\ i=j=l<k\ |\ j<i=k=l,\\
\\
          &\ \ \ \ \ i<j=k<l\ |\ i=k<j<l\ |\ i<k<j=l\ |\ j=l<i<k\ |\\
\frac12   &\text{if}\ j<l<i=k\ |\ j<i=l<k\ |\ i<j=l<k\ |\ l<i<j=k\ |\\
          &\ \ \ \ \ i<l<j=k\ |\ i=l<j<k\ |\ i=l<k<j\ |\ j<k=i<l,\\
\\
0         &\text{otherwise},
\end{cases}
\end{equation}
while arbitrary $\alpha$ and $\beta$ may be tackled by symmetry. Note that the value $2$ comes from two matching index pairs, and $1$ is obtained when three indices are equal. Beside the value $\frac{1}{2}$, all non-vanishing configurations are given, in which two indices are equal. It is worth to mention that these are only possible to realize when there are at least three different numbers. Hence the third case will appear only for $m\geq3$, explaining why $m=2$ has a smaller constant.

Likewise, one checks for any $\alpha \hat{=}(i,j)$, $\beta \hat{=}(k,l)$ in $S$,
\begin{equation}\label{eqno3}
\sum_{\gamma\in S}\<[\check{E}_\alpha,\check{E}_\gamma],[\check{E}_\beta,\check{E}_\gamma]\>=2n\delta_{ik}\delta_{jl}-2\delta_{ij}\delta_{kl},
\end{equation}
where $\<A,B\>=\operatorname{Re}\[\operatorname{tr}\(AB^*\)\]$.

Let $\{\check{Q}_\alpha\}_{\alpha\in S}$ be any orthonormal basis of $HM(n)$. Then there exists a unique
orthogonal matrix $Q\in O(N)$ such that
\begin{equation}\label{Qeq}
(\check{Q}_1,\cdots,\check{Q}_N)=(\check{E}_1,\cdots,\check{E}_N)Q,
 \end{equation}
 i.e., $\check{Q}_\alpha=\sum_\beta q_{\beta\alpha}\check{E}_\beta$ for $Q=(q_{\alpha\beta})_{N\times N}$, and if $\check{Q}_\alpha=(\check{q}_{ij}^\alpha)_{n\times n}$, $\gamma\hat{=}(i,j), \tau\hat{=}(j,i)$,
\begin{equation}\label{qalpha}
\check{q}_{ij}^\alpha=\overline{\check{q}_{ji}^\alpha}=
\begin{cases}
q_{\gamma\alpha}                                               &\text{if}\ i=j,\\
(q_{\gamma\alpha}-q_{\tau\alpha}\mathbf{i})/\sqrt{2}          &\text{if}\ i<j.
\end{cases}
\end{equation}
This change of basis utilizes again the interpretation of matrix operations on a row of non-scalar entries that has been seen before in the definition of the $K$ action.

Let $\lambda_1, \cdots, \lambda_n$ be $n$ real numbers satisfying $\sum_{i=1}^n\lambda_i^2=1$ and $\lambda_1\geq\cdots\geq\lambda_n$. Define
$I_1:=\{j\ |\ \lambda_1-\lambda_j>\frac2{\sqrt{3}}\}, I_2:=\{i\ |\ \lambda_i-\lambda_n>\frac2{\sqrt{3}}\}$, and $$I:=\left\{(i,j)\in S\ |\ \lambda_i-\lambda_j>\frac2{\sqrt{3}}\right\}.$$
Let $n_0$ be the number of elements of $I$. Then $({1}\times I_1)\bigcup(I_2\times {n})\subset I\subset S$. Now comes the first technical lemma. It looks pretty much like the variant from \cite{GT08} and \cite{Ge14}, but the sets $I_1$, $I_2$ and $I$ differ in the lower bound imposed to the difference.

\begin{lem}\label{lem1}
Either $I=\{1\}\times I_1$ or $I=I_2\times\{n\}$.
\end{lem}
\begin{proof}
If $n_0=0$, the three sets are all empty. If $n_0=1$, the single element must be $(1,n)\in S$, and the three sets are equal. If $n_0\geq2$, let $(1,n)$ and $(i_1,j_1)$ be two different elements of $I$, that is
$\lambda_1-\lambda_n\geq\lambda_{i_1}-\lambda_{j_1}>\frac2{\sqrt{3}}$ and $(1,n)\neq(i_1,j_1)$. We assert that either  $i_1=1$ and $j_1\neq n$ or $i_1\neq1$ and $j_1=n$, which shows exactly that $I=\{1\}\times I_1\cup I_2\times\{n\}$. Otherwise, $1$,$i_1$,$j_1$ and $n$ would be four different elements in $\{1,\cdots,n\}$, and thus
$$
1\geq\lambda^2_1+\lambda^2_{i_1}+\lambda^2_{j_1}+\lambda^2_n\geq\frac{1}2(\lambda_1-\lambda_n)^2+\frac{1}2(\lambda_{i_1}-\lambda_{j_1})^2>\frac{4}{3}
$$
is a contradiction. Next, without loss of generality, we assume  $(i_1,j_1)\in\{1\}\times I_1$. Then it will be seen that $I_2\times\{n\}=\{(1,n)\}$, and thus $I=\{1\}\times I_1$, which completes the proof. Otherwise, if there is another element, say $(i_2,n)$, in $I_2\times\{n\}$, then $i_2\neq j_1$ since otherwise we would get the following contradiction
$$2\geq 2(\lambda_1^2+\lambda_n^2)\geq (\lambda_1-\lambda_n)^2=(\lambda_1-\lambda_{j_1}+\lambda_{i_2}-\lambda_n)^2>\frac{16}{3}.$$
Hence $i_1=1$, $j_1$, $i_2$ and $n$ are four different elements in $\{1,\cdots,n\}$, and we come to the same contradiction as above.
\end{proof}

The following result is a straightforward consequence of the previous lemma. In \cite{GT08} and \cite{Ge14}, similar inequalities with respect to the differing $c$ were shown.
\begin{lem}\label{lem2}
We have $\sum_{(i,j)\in I}[(\lambda_i-\lambda_j)^2-\frac4{3}]\leq\frac{2}3,$ where the equality holds 
if and only if $n_0=1$ and $\lambda_1=-\lambda_n=\frac{\sqrt{2}}2,\lambda_2=\cdots\lambda_{n-1}=0$.
\end{lem}
\begin{proof}
Without loss of generality,we can assume $I=\{1\}\times I_1$ by Lemma \ref{lem1}. Then
$$
\begin{aligned}
\sum_{(i,j)\in I}\left[(\lambda_i-\lambda_j)^2-\frac43\right]&=\sum_{j\in I_1}(\lambda^2_1+\lambda^2_j-2\lambda_1\lambda_j)-\frac43n_0\\
                                                  &=n_0\lambda^2_1+\sum_{j\in I_1}\lambda^2_j-2\lambda_1\sum_{j\in I_1}\lambda_j-\frac43n_0\\
                                                  &\leq(n_0+1)\lambda^2_1+\sum_{j\in I_1}\lambda^2_j+\Big(\sum_{j\in I_1}\lambda_j\Big)^2-\frac43n_0\\
                                                  &\leq(n_0+1)\Big(\lambda^2_1+\sum_{j\in I_1}\lambda^2_j\Big)-\frac43n_0\\
                                                  &\leq(n_0+1)\sum^n_{i=1}\lambda^2_i-\frac43n_0\\
                                                  &=1-\frac13n_0\leq \frac23,
\end{aligned}
$$
where the equality condition is easily seen from the proof.
\end{proof}

The next lemma is adapted slightly from the symmetric case in \cite{GT08} to the Hermitian case. The difference comes from the index set $S$ which now includes not only those indices $\gamma\hat{=}(i,j)$ but also those $\tau\hat{=}(j,i)$ (when $i<j$). This leads to a difference in the expression (\ref{qalpha}) where $q_{\gamma\alpha}$ and $q_{\tau\alpha}$ are two entries in a column of the matrix $Q\in O(N)$.
\begin{lem}\label{lem3}
For any $Q\in O(N)$, given any $\alpha\in S$ and any subset $J_\alpha\subset S$, we have $$\sum_{\beta\in{J}_\alpha}\(\left\|\[\check{Q}_\alpha,\check{Q}_\beta\]\right\|^2-\frac{4}3\)\leq\frac{4}3.$$
\end{lem}
\begin{proof}
Since the required inequality is invariant under unitary congruences, i.e.,
$$\sum_{\beta\in{J}_\alpha}\left\|\[\check{Q}_\alpha,\check{Q}_\beta\]\right\|^2=\sum_{\beta\in{J}_\alpha}\left\|\[P^*\check{Q}_\alpha P,P^*\check{Q}_\beta P\]\right\|^2,\quad \text{for any~~ } P\in U(n),$$
we can assume without loss of generality $\check{Q}_\alpha=diag\(\lambda_1,\cdots,\lambda_n\)$ (Notice also that $(P^*\check{Q}_1P,\cdots,P^*\check{Q}_NP)$ is also an orthonormal basis of $HM(n)$ and can be expressed as (\ref{Qeq}) by a $Q\in O(N)$), with $\sum_i\lambda^2_i=1$, $\lambda_1\geq\cdots\geq\lambda_n$ and $I=\{1\}\times I_1$ due to Lemma \ref{lem1}. Then by (\ref{qalpha}) and Lemma \ref{lem2},
\begin{equation}\label{ineq-lem3}
\begin{aligned}
\sum_{\beta\in J_\alpha}\left(\left\|\[\check{Q}_\alpha,\check{Q}_\beta\]\right\|^2-\frac43\right)
&=\sum_{\beta\in J_\alpha}\(\sum_{i,j=1}^n\(\lambda_i-\lambda_j\)^2|\check{q}^\beta_{ij}|^2-\frac43\cdot1\)\\
&=\sum_{\beta\in J_\alpha}\sum_{(i,j)\hat{=}\gamma\in S}\(\(\lambda_i-\lambda_j\)^2-\frac43\)\cdot|\check{q}^\beta_{ij}|^2\\
&=\sum_{\beta\in J_\alpha}\sum_{(i,j)\hat{=}\gamma\in S}\(\(\lambda_i-\lambda_j\)^2-\frac43\)\cdot\frac12\(q^2_{\gamma\beta}+q^2_{\tau\beta}\)\\
&\leq2\sum_{(i,j)\hat{=}\gamma, i<j}\(\(\lambda_i-\lambda_j\)^2-\frac43\)\cdot\sum_{\beta\in J_\alpha}\frac12\(q^2_{\gamma\beta}+q^2_{\tau\beta}\)\\
&\leq2\sum_{(i,j)\hat{=}\gamma\in I}\(\(\lambda_i-\lambda_j\)^2-\frac43\)\cdot\sum_{\beta\in S}\frac12\(q^2_{\gamma\beta}+q^2_{\tau\beta}\)\\
&=2\sum_{(i,j)\hat{=}\gamma\in I}\(\(\lambda_i-\lambda_j\)^2-\frac43\)\cdot 1\leq\frac{4}3,
\end{aligned}
\end{equation}
where $\tau\hat{=}(j,i)$, the equality in the last line is because of $Q\in O(N)$ and negative summands coming from $(i,j)\not\in I$ were omitted.
\end{proof}

The last technical lemma bounds the sum of the commutator norms over the basis matrices. For the more restricted classes considered in \cite{GT08} and \cite{Ge14}, it was at most half as big.
\begin{lem}\label{lem4}
We have $\sum_{\beta\in S}\left\|\[\check{Q}_\alpha,\check{Q}_\beta\]\right\|^2\leq2n$ for any $ Q\in O(N)$ and any $\alpha\in S$.
\end{lem}
\begin{proof}
It follows from (\ref{eqno3}, \ref{Qeq}, \ref{qalpha}) that
$$
\begin{aligned}
\sum_{\beta\in S}\left\|\[\check{Q}_\alpha,\check{Q}_\beta\]\right\|^2
&=\sum_{\beta\gamma\tau\xi\eta}q_{\gamma\alpha}q_{\xi\alpha}q_{\tau\beta}q_{\eta\beta}\<[\check{E}_\gamma,\check{E}_\tau],[\check{E}_\xi,\check{E}_\eta]\>\\
&=\sum_{\gamma\xi}q_{\gamma\alpha}q_{\xi\alpha}\sum_\tau\<[\check{E}_\gamma,\check{E}_\tau],[\check{E}_\xi,\check{E}_\tau]\>\\
&=\sum_{\gamma\hat{=}(i,j),\xi\hat{=}(k,l)}q_{\gamma\alpha}q_{\xi\alpha}\cdot(2n\delta_{ik}\delta_{jl}-2\delta_{ij}\delta_{kl})\\
&=2n\sum_\gamma q^2_{\gamma\alpha}-2\Big(\sum_i\check{q}^\alpha_{ii}\Big)^2\\
&\leq 2n.
\end{aligned}
$$
\end{proof}

%

 \section{Proof of the main results}\label{sec3}
 Now, we are going to detail the method introduced in \cite{GT08}, with only slight modifications to use the critical lemmas of the previous section. In the first step, the inequality is rewritten in order to take profit of the matrix structure. As in \cite{GT08} and \cite{Ge14}, we transform the DDVV-type inequality to the non-positivity of a quadratic form $f_Q(x)$ $(x\in\mathbb{R}^N_+)$ parameterized by $Q\in SO(N)$. In order to do this, we need to introduce a ``multiplicative" map $\varphi$ as in \cite{GT08} and \cite{Ge14}. Here we also provide the readers with an alternative interpretation of $\varphi$ for a better understanding.

 Let $\varphi:M(m,n)\longrightarrow M(\binom{m}{2},\binom{n}{2})$ be the map defined by $\varphi\(A\)_{\(i,j\)\(k,l\)}:=A\binom{kl}{ij}$, where $1\leq i<j\leq m$, $1\leq k<l\leq n$, and $A\binom{kl}{ij}=a_{ik}a_{jl}-a_{il}a_{jk}$ is the discriminant of the $2\times2$ submatrix of $A$ that is the intersection of rows $i$ and $j$ with columns $k$ and $l$, arranged with the same order as in (\ref{eqno1}). One can verify directly that $\varphi\(I_n\)=I_{\binom{n}{2}}$, $\varphi\(A\)^t=\varphi\(A^t\)$, and in particular, the map $\varphi$ preserves the matrix product, i.e., $\varphi(AB)=\varphi(A)\varphi(B)$ holds for $A\in M(m,k)$ and $B\in M(k,n)$. Alternatively, one can regard $\varphi$ as a ``homomorphism"\footnote{It is indeed a homomorphism when $m=n$.}  from $M(m,n)\cong \mathrm{Hom}(\mathbb{R}^n, \mathbb{R}^m)$ to $M(\binom{m}{2},\binom{n}{2})\cong \mathrm{Hom}(\Lambda^2(\mathbb{R}^n), \Lambda^2(\mathbb{R}^m))$, by extending the linear map $A\in\mathrm{Hom}(\mathbb{R}^n, \mathbb{R}^m)$ naturally to a linear map $\varphi(A)$ from the exterior algebra (of grade $2$) $\Lambda^2(\mathbb{R}^n)$ to $\Lambda^2(\mathbb{R}^m)$:
$$\varphi(A)(x\wedge y):=(Ax)\wedge (Ay), \quad \text{for any~~}x,y\in \mathbb{R}^n.$$
An elementary proof via matrix was included in the first author's Bachelor thesis which was incorporated to the joint paper \cite{AGPP08}. However, in \cite{AGPP08} it was totally rewritten via exterior algebra (see Section 3 there) and in this viewpoint, the properties of $\varphi$ are well-known.

 Let $B_1,\cdots,B_m$ be any $n\times n$ Hermitian matrices ($n\geq2$). Their coefficients in the standard basis $\{\check{E}_\alpha|\alpha\in S\}$ of $HM(n)$ are determined by a matrix $B\in M\(N,m\)$ as
 \begin{equation}\label{BiB}
 \(B_1,\cdots,B_m\)=\(\check{E}_1,\cdots,\check{E}_N\)B.
 \end{equation}
 Since $B$ is real and $BB^t$ is a $N\times N$ positive semi-definite matrix, there
exists an orthogonal matrix $Q\in SO(N)$ such that
\begin{equation}\label{BBdiag}
BB^t=Q~diag(x_1,\cdots,x_N)~ Q^t, \quad \text{where}~~x_{\alpha}\geq 0,~ 1\leq\alpha\leq N.
\end{equation}
Thus
\begin{equation}\label{BBnorm}
\sum_{r=1}^m\|B_r\|^2=\|B\|^2 =\sum_{\alpha=1}^Nx_{\alpha}.
\end{equation}
Moreover, this orthogonal matrix $Q$ determines an orthonormal basis $\{\check{Q}_{\alpha}|\alpha\in S\}$ of $HM(n)$ as (\ref{Qeq}) in Section \ref{sec2}.

 We use the lexicographic order as in (\ref{eqno1}) for the indices sets $\{\(r,s\)|1\leq r<s\leq m\}$ and $\{\(\alpha,\beta\)|1\leq\alpha<\beta\leq N\}$. Then we can arrange $\{[B_r,B_s]\}_{r<s}$ and $\{[\check{E}_\alpha,\check{E}_\beta]\}_{\alpha<\beta}$ into $\binom{m}{2}$- and $\binom{N}{2}$-vectors, respectively.

Now we observe that
\begin{equation}\label{comm-bases}
\Big([B_1, B_2],\cdots,[B_{m-1}, B_m]\Big)=\Big([\check{E}_1, \check{E}_2],\cdots,[\check{E}_{N-1}, \check{E}_N]\Big) \varphi(B).
\end{equation}

 Let $C(\check{E})$ denote the matrix in $M(\binom{N}{2})$
defined by
\begin{equation}\label{comm-bases1}
C(\check{E})_{(\alpha,\beta)(\gamma,\tau)}:=\langle~
[\check{E}_{\alpha}, \check{E}_{\beta}],~ [\check{E}_{\gamma},
\check{E}_{\tau}]~ \rangle,
\end{equation}
for $1\leq\alpha<\beta\leq N$,
$1\leq\gamma<\tau\leq N$. Moreover we will use the same notation for
$\{B_r\}$ and $\{\check{Q}_{\alpha}\}$, \emph{i.e.}, the $\binom{m}{2}\times\binom{m}{2}$ matrix $C(B):=\Big(\langle~
[B_{r_1}, B_{s_1}],~ [B_{r_2}, B_{s_2}]~ \rangle\Big)$ and the $\binom{N}{2}\times\binom{N}{2}$ matrix $C(Q):=\Big(\langle~
[\check{Q}_{\alpha}, \check{Q}_{\beta}],~ [\check{Q}_{\gamma},
\check{Q}_{\tau}]~ \rangle\Big)$
respectively. Then it is obvious from (\ref{comm-bases}) that
\begin{equation}\label{comm-bases2}
C(B)=\varphi(B^t)C(\check{E})\varphi(B), \hskip 0.3cm C(Q)=\varphi(Q^t)C(\check{E})\varphi(Q).
\end{equation}

By (\ref{comm-bases2}) and the multiplicativity of $\varphi$, we have
\begin{eqnarray}\label{commtransf}
\sum_{r,s=1}^m\|[B_r, B_s]\|^2&=&2\operatorname{tr}~C(B)=2\operatorname{tr}\hskip 0.1cm
\varphi(B^t)C(\check{E})\varphi(B)\\
&=&2\operatorname{tr}~\varphi(BB^t)C(\check{E})=2\operatorname{tr}~
\varphi(diag(x_1,\cdots,x_N))C(Q)\nonumber\\
&=&\sum_{\alpha,\beta=1}^Nx_{\alpha}x_{\beta}\|[\check{Q}_{\alpha},
\check{Q}_{\beta}]\|^2.\nonumber
\end{eqnarray}

Combining (\ref{BBdiag}, \ref{BBnorm}, \ref{commtransf}), the inequality (1) of Theorem \ref{DDVVHermthm} is now transformed into the following:
\begin{equation}\label{DDVVtransf}
\sum_{\alpha,\beta=1}^Nx_\alpha x_\beta\left\|\[\check{Q}_\alpha,\check{Q}_\beta\]\right\|^2\leq\frac43\(\sum_{\alpha=1}^Nx_\alpha\)^2,\quad \forall x\in R^N_+,~ \forall Q\in SO\(N\),
\end{equation}
where $R^N_+:=\{x=\(x_1,\cdots,x_N\)\in R^N \mid x_\alpha\geq0,1\leq\alpha\leq N\}$.

The next step is to show (\ref{DDVVtransf}). For this, define the function:
\begin{equation}\label{quadratic-form}
f_Q(x)=F(x,Q):=\sum_{\alpha,\beta=1}^Nx_{\alpha}x_{\beta}\|[\check{Q}_{\alpha},
\check{Q}_{\beta}]\|^2-\frac{4}{3}\Big(\sum_{\alpha=1}^Nx_{\alpha}\Big)^2.
\end{equation}
Then $F$ is a continuous function defined on $\mathbb{R}^N\times
SO(N)$ (equipped with the product topology of Euclidean and Frobenius metric spaces) and thus uniformly continuous on any compact subset of
$\mathbb{R}^N\times SO(N)$. Let
$\Delta:=\{x\in\mathbb{R}^N_{+}~|~\sum_{\alpha}x_{\alpha}=1\}$
and for any sufficiently small $\varepsilon>0$,
$\Delta_{\varepsilon}:=\{x\in
\Delta~|~x_{\alpha}\geq \varepsilon, 1\leq\alpha\leq N\}$.
Also let
$$G:=\{Q\in SO(N)~|~f_Q(x)\leq 0, \text{~for~all~} x\in
\Delta\},$$
$$G_{\varepsilon}:=\{Q\in SO(N)~|~f_Q(x)< 0,
\text{~for~all~} x\in \Delta_{\varepsilon}\}.$$ We claim that
$G=\lim_{\varepsilon\rightarrow 0}G_{\varepsilon}=SO(N).$ Note that
this implies (\ref{DDVVtransf}) by the homogeneity of $f_Q$ and thus proves Theorem \ref{DDVVHermthm}. In
fact we can show
\begin{equation}\label{G epsilon}
G_{\varepsilon}=SO(N) ~\quad \text{for any sufficiently small } \varepsilon>0.
\end{equation}
To prove (\ref{G epsilon}), we use the continuity method, in which
we must prove the following three properties and remember that there are only the two trivial sets that are open and closed at the same time:
\begin{itemize}
\item[\textbf{(a)}]\label{step1} $I_N\in G_{\varepsilon}$ (and thus
$G_{\varepsilon}\neq\emptyset$);
\item[\textbf{(b)}]\label{step2} $G_{\varepsilon}$ is open in $SO(N)$;
\item[\textbf{(c)}]\label{step3} $G_{\varepsilon}$ is closed in $SO(N)$.
\end{itemize}

\textbf{Proof of (a)}.  For any $ x\in\Delta_\epsilon$, applying (\ref{eqno2}) we get
$$\begin{aligned}
  f_{I_N}\(x\) = & \sum_{\alpha,\beta=1}^Nx_\alpha{x}_\beta\|\[\check{E}_\alpha,\check{E}_\beta\]\|^2-\frac43\\
   =& ~4\sum_{i<j}x_{ij}x_{ji}+2\sum_{i<j}(x_{ii}x_{ij}+x_{ij}x_{jj}+x_{ii}x_{ji}+x_{ji}x_{jj})+\\
    & ~\sum_{i<j<k}(x_{ij}x_{jk}+x_{ij}x_{ik}+x_{ik}x_{jk}+x_{ji}x_{ki}
       +x_{ki}x_{kj}+x_{ji}x_{kj}+x_{ij}x_{kj}+\\
    & ~x_{jk}x_{ki}+x_{ik}x_{kj}+x_{ij}x_{ki}+x_{ik}x_{ji}+x_{ji}x_{jk})-\frac43\Big(\sum^n_{i,j=1}x_{ij}\Big)^2\\
\leq&  \sum_{i<j}\Big((x_{ij}+x_{ji})^2+2(x_{ii}+x_{jj})(x_{ij}+x_{ji})\Big)+\\
    &  \sum_{i<j<k}\Big( (x_{ik}+x_{jk}+x_{ki}+x_{kj})(x_{ij}+x_{ji})+(x_{ik}+x_{ki})(x_{jk}+x_{kj}) \Big)-\frac43\Big(\sum^n_{i,j=1}x_{ij}\Big)^2\\
   =& \sum_{i<j}\Big[(x_{ij}+x_{ji})\Big(x_{ij}+x_{ji}+\sum_{k=j+1}^n(x_{ik}+x_{jk}+x_{ki}+x_{kj})+\sum_{k=1}^{i-1}(x_{jk}+x_{kj})\Big)\Big]+\\
    & \sum_{i<j}\Big(2(x_{ii}+x_{jj})(x_{ij}+x_{ji})\Big)-\frac43\Big(\sum^n_{i<j}(x_{ij}+x_{ji})+\sum_kx_{kk}\Big)^2\\
   <& \Big(\sum_{i<j}(x_{ij}+x_{ji})\Big)\Big(\sum_{k<l}(x_{kl}+x_{lk})\Big)+2\Big(\sum_kx_{kk}\Big)\Big(\sum_{i<j}(x_{ij}+x_{ji})\Big)-\\
    &  \frac43\Big[\Big(\sum^n_{i<j}(x_{ij}+x_{ji})\Big)^2+2\Big(\sum_kx_{kk}\Big)\Big(\sum_{i<j}(x_{ij}+x_{ji})\Big)+\Big(\sum_kx_{kk}\Big)^2\Big]\\
   <& ~0,
\end{aligned}$$
  which means that $ I_N\in G_\varepsilon$. In fact, we have proven $f_{I_N}(x)<-\frac13$ for any $ x\in\Delta_\epsilon$. \hfill  $\qed$

 \textbf{Proof of (b)}.
 Since $F$ is uniformly continuous on
$\triangle_{\varepsilon}\times SO(N)$, the function $g(Q):=\max_{x\in\Delta_{\varepsilon}}F(x,Q)$ is continuous on $Q\in SO(N)$ and thus $G_{\varepsilon}$ is obviously open as the preimage of an open set $(-\infty, 0)$ of $g$. \hfill  $\qed$

\textbf{Proof of (c)}.
  We only need to prove the following \textbf{a priori estimate}: \textit{Suppose $f_Q\(x\)\leq0$ for every $x\in\Delta_\varepsilon$. Then $f_Q\(x\)<0$ for every $x\in\Delta_\varepsilon$.} Provided with this, for a sequence $\{Q_k\}\subset G_{\varepsilon}$ which converges to a $Q\in SO(N)$, we have
  $$g(Q)=\lim_{k\rightarrow\infty}g(Q_k)=\lim_{k\rightarrow\infty}\max_{x\in\Delta_{\varepsilon}}F(x,Q_k)\leq0.$$
  Therefore, $f_Q\(x\)\leq g(Q)\leq0$ for every $x\in\Delta_\varepsilon$. Then $f_Q\(x\)<0$ for every $x\in\Delta_\varepsilon$ and thus $Q\in G_{\varepsilon}$, proving the closedness of $G_{\varepsilon}$.

  The proof of this estimate is as follows: If there is a point $y\in\Delta_\varepsilon$ such that $f_Q\(y\)=0$, we can arrange the entries decreasingly and assume without loss of generality that for some $1\leq\gamma\leq N$,
  \begin{equation}\label{ydelta}
  y\in\Delta_\varepsilon^\gamma:=\{x\in\Delta_\varepsilon \mid x_\alpha>\varepsilon \text{ for } \alpha\leq\gamma,\text{ and } x_\beta=
  \varepsilon \text{ for } \beta>\gamma\}.
  \end{equation}
   With the given prerequisites, then $y$ is a maximum point of $f_Q\(x\)$ in the cone spanned by $\Delta_\varepsilon$ and an interior maximum point of $f_Q\(x\)$ in $\Delta_\varepsilon^\gamma$.
      Hence, applying the Lagrange Multiplier Method, there exist numbers
$b_{\gamma+1},\cdots,b_N$ and a number $a$ such that
\begin{equation}\label{partial f}
\begin{array}{ll}
\Big(\frac{\partial f_Q}{\partial x_1}(y),\cdots,\frac{\partial
f_Q}{\partial x_{\gamma}}(y)\Big)=2a(1,\cdots,1),&\\
\Big(\frac{\partial f_Q}{\partial
x_{\gamma+1}}(y),\cdots,\frac{\partial f_Q}{\partial
x_{N}}(y)\Big)=2(b_{\gamma+1},\cdots,b_N)&
\end{array}
\end{equation}
or equivalently
\begin{equation}\label{partial f 2}
\sum_{\beta=1}^Ny_{\beta}(\|[\check{Q}_{\alpha},
\check{Q}_{\beta}]\|^2)-\frac{4}{3}=\Big\{
\begin{array}{ll}
 a & \alpha\leq\gamma,\\
b_{\alpha}& \alpha>\gamma.
\end{array}
\end{equation}
Hence
$$
f_Q(y)=\Big(\sum_{\alpha=1}^{\gamma}y_{\alpha}\Big)a+\Big(\sum_{\alpha=\gamma+1}^Nb_{\alpha}\Big)\varepsilon
=0\quad \text{and} \quad
\sum_{\alpha=1}^{\gamma}y_{\alpha}+(N-\gamma)\varepsilon=1.
$$
Meanwhile, by the homogeneity of $f_Q$ we see $\frac{\partial f_Q}{\partial \nu}(y)=2(a\gamma
+\sum_{\alpha=\gamma+1}^Nb_{\alpha})\leq 0$, where
$\nu=(1,\cdots,1)$ is the vector normal to $\Delta$ in
$\mathbb{R}^N$. For any sufficiently small $\varepsilon$ (such as
$\varepsilon<1/N$), it follows from the last three formulas that
$a\geq 0$. Without loss of generality, we assume
$y_1=max\{y_1,\cdots,y_{\gamma}\}>\varepsilon$. Let $$J:=\left\{\beta\in
S~|~ \|[\check{Q}_{1}, \check{Q}_{\beta}]\|^2\geq \frac{4}{3}\right\},$$
and let $n_1$ be the number of elements of $J$.
 Now combining Lemma \ref{lem3}, Lemma \ref{lem4} and (\ref{partial f 2}) will give a
contradiction as follows:

\begin{equation}\label{contrad}
\begin{aligned}
\frac43\leq\frac43+a&=\sum^N_{\beta=2}y_\beta\left\|\[\check{Q}_1,\check{Q}_\beta\]\right\|^2\\
                    &=\sum_{\beta\in J}y_\beta\(\left\|\[\check{Q}_1,\check{Q}_\beta\]\right\|^2-\frac43\)+\frac43\sum_{\beta\in J}y_\beta+\sum_{\beta\in S\backslash J}y_\beta\left\|\[\check{Q}_1,\check{Q}_\beta\]\right\|^2\\
                    &\leq y_1\sum_{\beta\in J}\(\left\|\[\check{Q}_1,\check{Q}_\beta\]\right\|^2-\frac43\)+\frac43\sum_{\beta\in J}y_\beta+\sum_{\beta\in S\backslash J}y_\beta\left\|\[\check{Q}_1,\check{Q}_\beta\]\right\|^2\\
                    &\leq \frac43y_1+\frac43\sum_{\beta\in J}y_\beta+\sum_{\beta\in S\backslash J}y_\beta\left\|\[\check{Q}_1,\check{Q}_\beta\]\right\|^2\\
                    &\leq\frac43\sum^N_{\beta=1}y_\beta=\frac43.
\end{aligned}
\end{equation}
Thus $a=0$, the third line shows $y_{\beta}=y_1$ for $\beta\in J$ and the fourth line shows $\sum_{\beta\in J}\(\left\|\[\check{Q}_1,\check{Q}_\beta\]\right\|^2-\frac43\)=\frac43$, hence
\begin{equation}\label{n1N}
\sum_{\beta\in J}\|[\check{Q}_1,
\check{Q}_{\beta}]\|^2=\frac{4}{3}(n_1+1)\leq 2n<\frac{4}{3}N \quad (n\geq2).
\end{equation}
Hence $S\backslash(J\cup\{1\})\neq\emptyset$, and the last ``$\leq$" in
(\ref{contrad}) should be ``$<$" by the definition of $J$ and the
positivity of $y_{\beta}$ for $\beta\in S\backslash(J\cup\{1\})$.\hfill
$\Box$

Now we consider the equality condition of (1) of Theorem \ref{DDVVHermthm} in view of the proof of the a priori estimate. When $f_Q(y)=0$ for some $y\in\Delta$ and $Q\in SO(N)$, we have also (\ref{ydelta}-\ref{n1N}) with $\varepsilon=a=0$, all inequalities in (\ref{contrad}) and thus in Lemmas \ref{lem2} and \ref{lem3} achieve the equality. For the eigenvalues $\{\lambda_1,\ldots, \lambda_n\}$ of $\check{Q}_1$ in (\ref{contrad}), because of the equality condition of Lemma \ref{lem2}, $\lambda_1=-\lambda_n=\frac1{\sqrt{2}}$, $\lambda_2=\cdots=\lambda_{n-1}=0$, $n_0=1$ and $I=\{(1,n)\}$. Without loss of generality, let $\check{Q}_1=\frac1{\sqrt{2}}(E_{11}-E_{22})$ (replace the index $n$ with $2$ for simplicity and $I=\{(1,2)\}$ now). In the proof of Lemma \ref{lem3} where $J_{\alpha}=J$, the fifth line of (\ref{ineq-lem3}) shows $q_{\gamma\beta}=q_{\tau\beta}=0$ for $\gamma\hat{=}(1,2)\in I$, $\tau\hat{=}(2,1)$ and $\beta\in S\backslash J$; the fourth line of (\ref{ineq-lem3}) shows $q_{\gamma\beta}=q_{\tau\beta}=0$ for any $\gamma\in S\backslash \{(1,2),(2,1)\}$ and $\beta\in J$. Then we have $n_1=\#J=2$ since $Q$ is orthogonal (its column vectors and row vectors are orthonormal).
Moreover, for any $\beta\in J$, $\operatorname{rank}(\check{Q}_\beta)=2$, $\check{q}_{ij}^\beta=0$ for $(i,j)\in S\backslash \{(1,2),(2,1)\}$, and  $\check{q}_{12}^\beta=\overline{\check{q}_{21}^\beta}$ with norm $\frac{1}{\sqrt{2}}$. These yield that the two perpendicular matrices $\check{Q}_\beta$ ($\beta\in J$) should be in the form of $diag(H_2,0)$ and $diag(H_3,0)$ up to a rotation.
Notice also that by (\ref{contrad}), $y_{\beta}=0$ for $\beta\in S\backslash (J\cup\{1\})$ and $y_\beta=y_1=\frac{1}{3}$  for $\beta\in J$.

So far we have proven that  $f_Q(x)=0$ ($x\in\mathbb{R}^N_+$) if and only if in the orthonormal basis $\{\check{Q}_{\alpha}\}_{\alpha\in S}$ of $HM(N)$ determined by $Q\in O(N)$, there exist three of them, say $\check{Q}_1,\check{Q}_2,\check{Q}_3$, such that $x_1=x_2=x_3\geq0$, $x_{\alpha}=0$ for $3<\alpha\leq N$, and for some $P\in U(n)$, $P^*\check{Q}_iP=diag(H_i,0)$ for $i=1,2,3$ ($\lambda=\frac{1}{\sqrt{2}}$ in $H_i$) up to a simultaneous rotation of the three matrices.

Now let $B_1,\cdots,B_m$ ($m\geq3$) be Hermitian matrices achieving the equality with the corresponding matrices $B\in M(N,m)$, $Q\in O(N)$ satisfying (\ref{BiB}, \ref{Qeq}, \ref{BBdiag}) and the equality conditions in the last paragraph. Then we have $Q^tBB^tQ=diag(cI_3,0)$ for some $c>0$ ($c=0$ only if $B_i=0$ for all $i=1,\ldots,m$). Hence
$Q^tB=\sqrt{c}(a_1,a_2,a_3,0,\cdots,0)^t\in M(N,m)$ for some orthonormal column vectors $a_i$'s in $\mathbb{R}^m$.
Thus there is an orthogonal matrix $R\in O(m)$ (by extending $a_i$'s to an orthonormal basis $\{a_i|i=1,\ldots,m\}$ of $\mathbb{R}^m$ and taking $R=(a_1,\ldots,a_m)\in O(m)$) such that
$$Q^tBR=\sqrt{c}\begin{pmatrix}
I_3 & 0 \\
0 & 0 \\
\end{pmatrix}\in M(N,m).$$
It follows from (\ref{BiB}, \ref{Qeq}) that
$$\(B_1,\cdots,B_m\)R=\(\check{E}_1,\cdots,\check{E}_N\)BR=\(\check{Q}_1,\cdots,\check{Q}_N\)Q^tBR=\sqrt{c}\(\check{Q}_1,\check{Q}_2,\check{Q}_3,0,\cdots,0\).$$
Recalling the properties of $\check{Q}_i$'s ($i=1,2,3$), we have completed the proof of the equality condition of (1) of Theorem \ref{DDVVHermthm}.

\begin{ack}
The authors are very grateful to Professor Zhiqin Lu for sending the preprint paper \cite{LW11} and valuable discussions.  Many thanks are due to the referee for helpful comments which greatly improved the presentation. The first author would also thank Tsinghua Sanya International Mathematics Forum where a part of the work was done during his attendance in the Young Geometric Analysts' Forum.
\end{ack}


\end{document}